\documentclass[12pt]{amsart}

\usepackage{color}
\usepackage{amsmath}
\usepackage{amssymb}
\usepackage{amsthm}
\usepackage{enumerate}
\usepackage{amsbsy}
\usepackage{amsfonts}
\newtheorem{theo}{Theorem}[section]

\newtheorem{coro}[theo]{Corollary}
\newtheorem{lemm}[theo]{Lemma}

\newcommand{\Ga}{\Gamma}

\newcommand{\Om}{\Omega}

\newcommand{\si}{\sigma}

\newcommand{\te}{\theta}
\newcommand{\De}{\Delta}
\newcommand{\de}{\delta}

\newcommand{\pa}{\partial}

\newcommand{\R}{{\bf R}^n}

\newcommand{\ri}{\rightarrow}
\newcommand{\Rn}{{\bf R}^{n-1}}

\newcommand{\na}{\nabla}
\newcommand{\RN}{{\bf R}^{n+1}}
\begin{document}
\baselineskip=18pt

\title[Maximum modulus for Stokes equations]{Maximum modulus estimate for the solution of the nonstationary Stokes equations}
\author{TongKeun Chang and Hi Jun Choe}
\address{TongKeun chang: Department of
Mathematics, Yonsei University, 50 Yonsei-ro, Seodaemun-gu, Seoul,
South Korea 120-749 } \email{chang7357@yonsei.ac.kr }
\address{Hi Jun Choe: Department of
Mathematics, Yonsei University, 50 Yonsei-ro, Seodaemun-gu, Seoul,
South Korea 120-749 } \email{choe@yonsei.ac.kr }

\thanks{}

\begin{abstract}
A maximum modulus estimate for the nonstationary Stokes equations in $C^2$ domain is found.
The singular part and regular part of Poisson kernel are analyzed.
The singular part consists of the gradient of single layer potential and  the gradient of composite potential
defined on only normal component of the boundary data.
Furthermore, the normal velocity near the boundary is bounded if the boundary data is bounded.
If the normal component of the boundary data is Dini-continuous
and the tangential component of the boundary data is bounded,
then the maximum modulus of velocity is bounded in whole domain.
\end{abstract}

\maketitle


\section{Introduction}
\setcounter{equation}{0}

A maximum modulus estimate of the nonstationary Stokes equations is presented. In the case of the stationary flow, Maremonti and Russo\cite{MRu} obtained a quasi maximum principle and Varnhorn\cite{V} showed a maximum modulus theorem for $C^{1,\alpha}$ domain:
$$
\max_{x\in \Omega} |u| \leq C(\Omega) \max_{x\in \partial\Omega} |u|,
$$
where $u$ is a solution to the stationary Stokes equations in domain $\Omega$. We also note that Maz'ja and Rossmann\cite{MRo} considered the maximum modulus estimate for the stationary Navier-Stokes equations in polygonal domain.
In a canonical domain like ball, Kratz\cite{KR} found the best constant $C(\Omega)$ such that
$$
\max_{x\in B_1} |u| \leq \frac{1}{2}n(n+1) \max_{x\in \partial B_1} |u|,
$$
where $B_1$ is the unit ball in ${\bf R}^n$.
For a more general domain like Lipschitz in ${\bf R}^3$, Shen\cite{Sh2} obtained a maximum modulus estimate and the higher dimension problem is still unresolved.

The maximum modulus estimate of the nonstationary problem is heavily entangled with the structural form of  Poisson kernel and the solvability of the boundary value problem is essential. As a classical result, Solonnikov\cite{So} solved the initial-boundary problem in $C^2$ domain for the isotropic Sobolev spaces and later  he\cite{So1} extended the solvability to the anisotropic Sobolev spaces.
The $L^2$ solvability for the Lipschitz domain was obtained by Shen\cite{Sh1} for any dimension and Choe and Kozono\cite{Ch} considered the case for the mixed norm potential spaces.

To be more specific, we state the nonstationary Stokes equations:
\begin{align}\label{maineq}
\begin{array}{ll}\vspace{2mm}
u_t - \nu \De u + \na p =0 &  \mbox{ in } \Omega \times (0,T),\\ \vspace{2mm}
div \, u =0 &  \mbox{ in } \Omega\times (0,T),\\ \vspace{2mm}
u|_{t=0} = 0 &  \mbox{ in } \Omega, \\
u|_{\pa \Om \times (0,T)}= g   & \mbox{ on } \partial\Omega \times (0,T),
\end{array}
\end{align}
where $\Om $ is $C^2$ bounded connected domain in $\R$ and $0< T < \infty$ and $\nu$ is the viscosity which we assume 1.
In addition, we assume the boundary data $g$ satisfies the compatibility condition:
$$
\int_{\partial\Om} g\cdot N d\sigma = 0
$$
for almost all $t$, where $N$ is the outward unit normal vector on the boundary.
Since nontrivial initial data can be treated by solving homogeneous boundary value problems, we consider only the initial-boundary value problems with zero initial data.

Contrary to the stationary case, the quasi maximum principle fails, namely, there is an unbounded solution whose boundary data is bounded. Heuristically speaking, at the boundary point where the normal component of boundary data has a jump discontinuity along an $(n-2)$-dimensional surface on the boundary passing to it, the tangential component of the velocity blows up in the neighborhood of it. So we can not expect  the quasi maximum modulus theorem like the stationary case.

In this paper, we only consider the case that the space dimension is greater than or equal to 3. Dimension 2 case follows exactly the same path with logarithmic kernels.
Denote $E$ for the fundamental solution to Laplace equation and $\Ga$ for the fundamental solution to heat equation with unit conductivity. For a given boundary point $y \in\partial \Omega$ $N(y)$ is the outward unit normal vector at $y$. We define the $(n-1)$-dimensional convolution
$$
{\bf S} (f)(x) = \int_{\partial \Omega} E(x-y) f(y) d\sigma(y)
$$
for real-value function $f:\R \ri {\bf R}$ which is just the single layer potential of $f$ on $\partial\Omega$. We need a composite kernel.
We define a composite kernel function $\kappa(x,t)$ on $\Om\times (0,T)$ by
\begin{align*}
\kappa (x,t) = & \int_{\pa \Om}    \frac{\partial\Ga}{\partial N(y)}(x-y,t)    E (y)d\sigma(y)
\end{align*}
and a surface  potential ${\bf  T}$ for $f$ by
$$
{\bf T}(f)(x,t) =4 \int_{0}^{t }\int_{ \partial\Om} \kappa(x-y,t-s)  f(y,s) d\sigma(y) ds,
$$
for real-value function $f:\RN \ri {\bf R}$.
We state our main theorem:
For given $x\in \Om$, $\bar{x}$ is the nearest point of $x$ on $\partial\Om$ such that $dist(x,\partial\Om)=|\bar{x}-x|$
and for a vector valued function $v(x)$, we define the normal component and tangential component to the nearest point $\bar{x}$ by
$$
v_N(x) = (v(x)\cdot N(\bar{x})) N(\bar{x}) \quad \mbox{and}\quad v_T(x)= v(x) -  (v(x)\cdot N(\bar{x})) N(\bar{x}).
$$
\begin{theo}\label{theorem1}
Suppose that the domain $\Omega$ is bounded $C^2$ and $u$ is a solution to (\ref{maineq})
for bounded boundary data $g$.
The normal component of the velocity $u_N$ is bounded
and  there is also a constant
$C(\Om)$ such that
$$
\max_{(x,t)\in \Omega\times (0,T)} |u_N (x,t)|
 \leq C(\Omega) \max_{(y,t)\in \partial\Omega\times (0,T) } |g(y,t)|.
$$
Furthermore, the tangential component of the velocity $u$ satisfies that
$$
\max_{(x,t)\in \Omega\times (0,T)} |u_T(x,t)-\nabla {\bf S}(g \cdot N)_T (x,t) -\nabla {\bf T}(g \cdot N)_T (x,t) |
 \leq C(\Omega) \max_{(y,t)\in \partial\Omega\times (0,T) } |g(y,t)|
$$

\end{theo}
Define the modulus of continuity of $f$ at $x$  by $\omega(f)(r,x)=\sup_{y\in B_r(x)\cap \Omega} |f(y)-f(x)|$ and we say $f$ is Dini-continuous in $\Omega$  if
$$
|| f||_{Dini,\Omega} =\sup_{x\in \Omega} \int_{0}^{r_0} \omega(f)(r,x) \frac{dr}{r} < \infty
$$
for an $r_0>0$.
From a direct computation, we have $\nabla{\bf S}(f)$ and  $\nabla {\bf T}(f)$ are bounded if $f$ is Dini-continuous on $\partial \Omega$ and
we obtain a maximum modulus estimate:
\begin{coro}\label{theorem2}
Suppose that the domain $\Omega$ is bounded $C^2$ and $u$ is a solution to (\ref{maineq}).
Suppose $g$ is bounded on $\partial\Omega$ and the normal component $g_N$ is Dini-continuous.
Then, there is a constant $C(\Om)$ depending only on $\Omega$ such that
$$
\max_{(x,t)\in \Omega\times (0,T) } |u(x,t) | \leq C( \Omega) ( \max_{(y,t)\in  \partial\Omega\times (0,T)}
 |g(y,t)| +||g_N||_{Dini,\partial\Omega} ).
$$
\end{coro}
 As a separate interest, we obtain an improved $L^2$ theory like Lemma 4.1.  When the $L^2(\partial\Om)$ norm of the boundary data is bounded in time, then $||u(\cdot,t)||_{L^2(\Om)}$  is bounded in time. Consequently, the local boundedness holds too(see Corollary 4.2.).

\section{Kernels on  half plane}
\setcounter{equation}{0}

To study the equation \eqref{maineq}, we consider the case of $\Om = \R_+ = \{ (x^{\prime},x_n) \in \R \, | \, x^{\prime} \in \Rn, \,\, 0 < x_n < \infty \}$ and for the notational simplicity we set $D_{x_i}=\frac{\partial}{\partial x_i}$ and double indices means summation up to $n$. For notation,  we denote $x=(x^{\prime},x_n)$, that is, $x^{\prime}=(x_1,x_2,\cdot\cdot\cdot,x_{n-1})$. Indeed, the symbol $ {\prime}$ means the coordinate up to $n-1$ and
 $\omega_n$ is the surface area of the unit sphere in ${\bf R}^n$.

 We let $\Gamma$ be the fundamental solution to the heat equation such that
\begin{align*}
\Gamma(x,t)  = \left\{\begin{array}{ll} \vspace{2mm}
\frac{1}{\sqrt{2\pi t}^n} e^{ -\frac{|x|^2}{2t} }, &   t > 0\\
  0, & t \leq 0
\end{array}
\right.
\end{align*}
and $H$ be the Newtonian potential of $\Ga$  such that
$$
H(x,t) = \int_{{\bf R}^n} \Ga(y,t) E(x-y) dy.
$$

The Stokes fundamental matrix $(F,\gamma)$ for ${\bf R}^n, n \geq 3$  is
\begin{align*}
F_{ij}(x,t) &= \delta_{ij} \Ga(x,t)  + \frac{1}{(n-2)\omega_n} D_{x_i}D_{x_j} H(x,t) \\
\gamma_i &= \frac{ \delta(t)}{\omega_n} \frac{x_i}{ |x|^n},
\end{align*}
where $\delta(t)$ is the Dirac delta function and $\delta_{ij}$ is the Kronecker delta function.

The Green's matrix $(G,\theta)$ for the half space ${\bf R}^n_{+}$ is
\begin{align*}
G_{ij}(x,y,t) &= \delta_{ij}\big( \Ga(x-y,t) -\Ga(x-y^*,t)\big)\\
& +4(1-\delta_{jn}) D_{x_j} \int_{0}^{x_n} \int_{{\bf R}^{n-1}} D_{x_i} E(x-z) \Ga(z-y^{*},t) dz\\
\theta_j(x,y,t) &=(1-\delta_{jn})\left( \int_{{\bf R}^{n-1}} D_{x_i} E(x^{\prime}-z^{\prime},x_n)
 \Ga(z^{\prime}-y^{\prime},y_n,t) dz^{\prime}\right.\\
&\left. +  \int_{{\bf R}^{n-1}}E(x^{\prime}-z^{\prime},x_n)  D_{y_n}
 \Ga(z^{\prime}-y^{\prime},y_n,t) dz^{\prime}\right) ,
\end{align*}
where we denote $x^{*}=(x^{\prime},-x_n)$.

The Poisson kernel $(K,\pi ) $ for the half space is defined by
\begin{align*}
\begin{array}{ll}\vspace{2mm}
K_{ij}(x'-y',x_n,t) & = \frac{\pa G_{ij}(x,y,t)}{\pa y_n}|_{y_n =0} - \de_{jn} \te_i(x,y,t)|_{y_n =0}\\ \vspace{2mm}
& =  -2 \delta_{ij} D_{x_n}\Ga(x'-y', x_n,t)  +4L_{ij} (x'-y',x_n,t) \\ \vspace{2mm}
& \quad-  \de_{jn} \de(t)  D_{x_i} E(x'-y',x_n),\\ \vspace{2mm}
\pi_j (x'-y',x_n,t )&=-2\delta(t) D_{x_j}D_{x_n}E(x'-y',x_n)+4 D_{x_n}D_{x_n}A(x'-y',x_n,t)\\
&\quad+4D_{t} D_{x_j}A(x'-y',x_n,t),
\end{array}
\end{align*}
where we defined  that
\begin{align*}
 {L}_{ij} (x,t) & =  D_{x_j}\int_0^{x_n} \int_{\Rn}   D_{z_n}    \Ga(z,t) D_{x_i}   E(x-z)  dz,\\
 A(x,t)&=\int_{\Rn}\Ga(z^{\prime},0,t)E(x^{\prime}-z^{\prime},x_n)dz^{\prime}.
\end{align*}

$L_{ij}$ and $A$ satisfy the estimates
\begin{align}\label{inequality11}
|D^{l_0}_{x_n} D^{k_0}_{x'} D_{t}^{m_0} L_{ij}(x,t)|
& \leq  \frac{c}{t^{m_0 + \frac12} (|x|^2 +t )^{\frac12 n + \frac12 k_0} (x_n^2 +t)^{\frac12 l_0}},\\
|D^j_xD^m_tA(x,t)|&\leq \frac{c}{t^{m+\frac{1}{2}}(|x|^2+t)^{\frac{n-2+|j|}{2}}},
\end{align}
where $ 1 \leq  i \leq n$ and $1 \leq j \leq n-1$ (see \cite{K} and \cite{So}).
The estimates \eqref{inequality11} of $L_{ij}$ and the estimate of the fundamental solution to heat equation $\Gamma$ imply that
\begin{align}\label{inequality1}
|D^{l_0}_{x_n} D^{k_0}_{x'} D_{t}^{m_0} K_{ij}(x,t)|
& \leq  \frac{c}{t^{m_0 + \frac12} (|x|^2 +t )^{\frac12 n + \frac12 k_0} (x_n^2 +t)^{\frac12 l_0}}.
\end{align}

The solution $(u, p)$ of the Stokes system \eqref{maineq} in $\Om = {\bf R}^n_+$ with boundary data $g$ is expressed by
\begin{align}\label{representation}
\begin{array}{ll}\vspace{2mm}
u^i(x,t) &= \sum_{j=1}^{n}\int_0^t \int_{\Rn}
K_{ij}( x^{\prime}-y^{\prime},x_n,t-s)g_j(y^{\prime},s) dy^{\prime}ds,\\
p(x,t) &= \sum_{j=1}^{n}\int_0^t \int_{\Rn} \pi_j(x^{\prime}-y^{\prime},x_n,t-s) g_j(y^{\prime},s) dy^{\prime}ds.
\end{array}
\end{align}

We have relations among $L$ and $A$ such that
\begin{align} \label{1006-3}
\sum_{1 \leq i \leq n} L_{ii} = -2D_{x_n} \Ga, \qquad  L_{in} = L_{ni}  + B_{in},
\end{align}
where $ B_{in}(x,t) =  \int_{{\bf R}^{n-1}}D_{x_n}
\Ga(x^\prime -y^\prime , x_n, t) D_{y_i} E( y^{\prime},0) dy^\prime =  \frac{\partial}{x_i} \kappa (x,t)
$ if $ i\neq n$ and $B_{nn}=0$.

For further computation, we introduce Lebesgue spaces and Sobolev spaces:
$$
L^p (\Om) = \{ f ; \int_{\Om} |f|^p dx < \infty\}, \quad
W^{1,p}(\Om)=\{ f; \int_\Om |f|^p +|\nabla f|^p dx < \infty\},
$$
$$
L^p(0,T;W^{1,p}(\Om)) =\{f ; \int_0^T  \int_\Om |f|^p +|\nabla f|^p dx  dt < \infty\}.
$$

\section{Maximum modulus estimate in the half space}
\setcounter{equation}{0}
In this section, we consider the maximum modulus estimate in the half space. The normal derivative $D_{x_n}\Gamma$
has uniformly bounded $L^1$ norm with respect to $x_n$ on $\partial{\bf R}^{n}_{+} \times (0,T)$(see \eqref{1006-2}) and hence we focus only on the kernel function $L_{ij}$. By introducing a composite kernel $\kappa$ we are able to identify the singular kernels.
The following lemma is a key stone for the maximum modulus estimate.
\begin{lemm}\label{L-1-norm}
Let $1 \leq i \leq n$ and $1 \leq j \leq n-1$. Then
\begin{align}\label{0126-1}
\int_0^\infty \int_{{\bf R}^{n-1}} |L_{ij}(x',x_n,t)| dx'dt < C,
\end{align}
where $C>0$ is independent of $x_n >0$ and hence  it follows that
\begin{align}\label{0126-2}
\int_0^\infty \int_{{\bf R}^{n-1}} |L_{in}(x',x_n,t) - B_{in}(x^\prime,x_n,t)| dx'dt < C,
\end{align}
where $C>0$ is independent of $x_n >0$.
\end{lemm}
The maximum modulus theorem for the half space follows from Lemma \ref{L-1-norm}.
\begin{theo}
Let $g= (g_1, g_2, \cdots,  g_n) \in L^\infty(\partial{\bf R}^{n}_{+} \times ( 0, T ))$ and $(u,p)$ is
represented by \eqref{representation}. Then,
\begin{align}
\| u_T - \nabla {\bf S}_T (g_n) - \nabla {\bf T}_T (g_n) \|_{L^\infty (\R_+ \times (0, T))} \leq C \| g\|_{L^\infty(\partial{\bf R}^{n}_{+} \times (0,T))}
\end{align}
for some $C> 0$.
Furthermore, the normal component of the velocity $u$ is bounded
and  there is also a constant
$C$ such that
$$
\max_{(x,t)\in {\bf R}^n_+\times (0,T)} |u_n(x,t)|
 \leq C \max_{(y,t)\in \partial{\bf R}^{n}_+ \times (0,T) } |g(y,t)|.
$$
\end{theo}

To show the $L^1$ boundedness of $L_{ij}$, we note that
\begin{align} \label{0920}
 L_{ij}   (x,t) =2^\frac32 \pi^\frac12 \int_0^{x_n} t^{-\frac32} y_n e^{-\frac{|y_n|^2}{t}}
 \int_{{\bf R}^{n-1}}  D_{y_j}   \Ga^{\prime}(y^{\prime},t)
 D_{y_i} E(x^{\prime}-y^{\prime}, x_n - y_n)dy^{\prime} dy_n,
\end{align}
where  $\Ga^{\prime}$ is Gaussian kernel in $\Rn$.

\begin{lemm}\label{lemma1}
For $1 \leq j \leq n-1$, we get
\begin{align}\label{0119-1}
\begin{array}{l}  \vspace{2mm}
|\int_{|x'-y'| \leq \frac12 |x'|} D_{y_j}\Ga^{\prime} (y',t)  D_{y_n} E(x'-y',y_n)  dy'|
 \leq C t^{-\frac{n-1}2} e^{-\frac{|x'|^2}t } |x'|^{-1}
         + C  t^{-\frac{n}2 -\frac12}|x' |e^{-\frac{|x'|^2}{t}}\\ \vspace{2mm}
|\int_{\frac12 |x'| \leq |y'| \leq 2|x'|, |x'-y'| \geq \frac12 |x'|} D_{y_j}
                 \Ga^{\prime} (y',t)  D_{y_n} E(x'-y',y_n) dy'|
 \leq C t^{-\frac{n}2 -\frac12}|x'| e^{-\frac{|x'|^2}{t}},\\ \vspace{2mm}
|\int_{|y'| \leq \frac12 |x'|}D_{y_j}\Ga^{\prime} (y',t)  D_{y_n} E(x'-y',y_n) dy'|
\leq  C |x'|^{-n} \int_{|y'| \leq \frac12\frac{|x'|}{\sqrt{t}}}   |y'|^2 e^{-|y'|^2} dy'\\
             |\int_{|y'| \geq 2 |x'|}  D_{y_j}\Ga^{\prime} (y',t)  D_{y_n} E(x'-y',y_n)  dy'|
\leq C t^{-\frac{n}2  } \int_{\frac{2 |x'|}{\sqrt{t}} \leq |y'| }|y'|^{-n +2}  e^{-|y'|^2} dy',
\end{array}
\end{align}
where $C> 0$ is independent of $x', y_n$ and $t$.
\end{lemm}
\begin{proof}
Using integration by parts, we get
\begin{align}\label{0919}
& \notag \int_{|x'-y'| \leq \frac12 |x'|} D_{y_j}\Ga^{\prime} (y',t)  D_{x_n} E(x'-y',y_n)  dy'\\
 & \qquad  =  \int_{|x'- y'| = \frac12 |x'|}  \frac{x_j  - y_j}{|x' -y'|}\Ga^{\prime} (y',t)
          D_{y_n}E(x'-y',y_n) \si( dy')\\
  & \notag \qquad   \qquad
              -\int_{|x'-y'| \leq \frac12 |x'|}\Ga^{\prime} (y',t)  D_{y_j}D_{y_n}E(x'-y',y_n)  dy'.
\end{align}
For $y'$ with $|x' - y'| = \frac12 |x'|$, we get
$ | \Ga^{\prime} (y',t)  | \leq C t^{-\frac{n-1}2} e^{-\frac{|x'|^2}{2t}}$ and
$| D_{y_n}E(x'-y',y_n)|  \leq C\frac{1}{(|x'|^2 + y_n^2)^{\frac{n-1}2}} $.
Here, the first term of the right hand side in \eqref{0919} is dominated by
\begin{align}\label{0119}
\begin{array}{ll}\vspace{2mm}
  \int_{|x'-y'| =\frac12 |x'|}    |\Ga^{\prime}  (y',t)| |D_{y_n} E(x'-y',y_n)| \si (  dy')
 &\leq C t^{-\frac{n-1}2} e^{-\frac{|x'|^2}t } \frac{|x'|^{n-2}}{(|x'|^2 + y_n^2)^{\frac{n-1}2}}\\
 & \leq C t^{-\frac{n-1}2} e^{-\frac{|x'|^2}t } |x'|^{-1}.
 \end{array}
 \end{align}
 Since $ \int_{| x'-y'| \leq \frac12 |x'|}    D_{y_j }D_{x_n} E(x'-y',y_n)   dy' =0$,
 using the Mean value theorem,  the second term of  the right hand side of \eqref{0919} satisfies
\begin{align} \label{0119-2}
\begin{array}{ll} \vspace{2mm}
%&\int_{| x'-y'| \leq \frac12 |x'|} \Ga^{\prime} (y',t) D_{y_j }D_{x_n} E(x'-y',y_n)   dy'\\ \vspace{2mm}
 &  \int_{|x'- y'| \leq \frac12 |x'|} ( \Ga^{\prime} (y',t)
            -\Ga^{\prime} (x',t) ) D_{y_j} D_{x_n}E(x'-y',y_n)  dy'\\ \vspace{2mm}
& \leq C  |x' | t^{-\frac{n-1}2 -1}  e^{-\frac{|x'|^2}t}
           \int_{|x'- y'| \leq \frac12 |x'|}
           \frac{|x'-y'|^2 y_n}{(|x'-y'|^2 + y_n^2)^{\frac{n}2 +1} }  dy'\\
           \vspace{2mm}
& \leq C  |x' | t^{-\frac{n-1}2 -1}  e^{-\frac{|x'|^2}t}
            \int_{\Rn}  \frac{1}{(|y'|^2 + 1)^{\frac{n}2}} dy'.
\end{array}
\end{align}
By \eqref{0919} - \eqref{0119-2}, we obtain $\eqref{0119-1}_1$.

For  $\eqref{0119-1}_2$, note that for $y'$ satisfying
$\frac12 |x'| \leq |y'| \leq 2|x'|$ we have $ |x'-y'| \geq \frac12 |x'|$.  We have
$| D_{y_j} \Ga^{\prime} (y',t)| \leq C t^{-\frac{n-1}2 -1} |x'| e^{-\frac{|x'|^2}{t}}$ and
$D_{y_n}E(x'-y',y_n) \leq  C  |x'|^{-\frac{n-1}2} $, and thus we get
 \begin{align*}
\int_{\frac12 |x'| \leq |y'| \leq 2|x'|, |x'-y'| \geq \frac12 |x'|}
     D_{y_j}\Ga^{\prime} (y',t) D_{y_n}E(x'-y',y_n)  dy'
& \leq C  t^{-\frac{n}2 -\frac12} |x'|   e^{-\frac{|x'|^2}{t}}.
\end{align*}
Hence, we obtain  $\eqref{0119-1}_2$.

Since $\int_{|y'| \leq \frac12 |x'|}  D_{y_j}\Ga^{\prime} (y',t)    dy  =0$,
using Mean-value Theorem $, \eqref{0119-1}_3$ is proved by
\begin{align*}
 & \int_{|y'| \leq \frac12 |x'|}  D_{y_j}\Ga^{\prime} (y',t)
           \Big(   D_{x_n} E(x'-y',y_n) -  D_{x_n} E(x',y_n) \Big) dy'\\
& \leq C  (|x'|^2 + y_n)^{ -\frac{n}2} \int_{|y'| \leq \frac12|x'|}
t^{-\frac{n+1}2} |y'|^2 e^{-\frac{|y'|^2}t}dy' \\
& \leq C (|x'|^2 + y_n)^{ -\frac{n}2} \int_{|y'| \leq \frac12\frac{|x'|}{\sqrt{t}}}   |y'|^2 e^{-|y|^2} dy'.
\end{align*}
Finally, $\eqref{0119-1}_4$  follows by
\begin{align*}
\int_{|y'| \geq 2 |x'|} D_{y_j}\Ga^{\prime} (y',t)D_{y_n} E(x'-y',y_n)  dy'
 &\leq C t^{-\frac{n +1}2  } \int_{ 2 |x'|  \leq |y'| }
          |y'|^{-n +2}  e^{-\frac{|y'|^2}{t}} dy'\\
 & =C t^{-\frac{n}2  } \int_{\frac{2 |x'|}{\sqrt{t}} \leq |y'| }
           |y'|^{-n +2}  e^{-|y'|^2} dy'.
\end{align*}
\end{proof}

Following a similar proof to Lemma 3.3, we get the following lemma.
\begin{lemm}\label{lemma2}
For $1 \leq i, \,\, j \leq n-1$, we get
\begin{align*}
&\int_{|x'-y'| \leq \frac12 |x'|} D_{y_j}\Ga^{\prime} (y',t)  D_{y_i} E(x'-y',y_n)  dy'
 \leq C t^{-\frac{n}2 -\frac12}|x' |e^{-\frac{|x'|^2}{t}}\\
&\int_{\frac12 |x'| \leq |y'| \leq 2|x'|, |x'-y'| \geq \frac12 |x'|}
            D_{y_j}\Ga^{\prime} (y',t)  D_{y_i} E(x'-y',y_n) dy'
 \leq C t^{-\frac{n}2 -\frac12}|x'| e^{-\frac{|x'|^2}{t}},\\
&\int_{|y| \leq \frac12 |x'|}D_{y_j}\Ga^{\prime} (y',t)  D_{y_i} E(x'-y',y_n) dy'
 \leq C |x'|^{-n} \int_{|y'| \leq C \frac12\frac{|x'|}{\sqrt{t}}}   |y'|^2 e^{-|y'|^2} dy',\\
& \int_{|y'| \geq 2 |x'|}  D_{y_j}\Ga^{\prime} (y',t)  D_{y_i} E(x'-y',y_n)  dy' \leq C t^{-\frac{n}2  }
\int_{\frac{2 |x'|}{\sqrt{t}} \leq |y'| } |y'|^{-n +2}  e^{-|y'|^2} dy'.
\end{align*}
\end{lemm}

\noindent
{\bf Proof of Lemma \ref{L-1-norm}.}

Note that for $1 \leq i \leq n$ and $1 \leq j \leq n-1$
\begin{align*}
\int_0^T  \int_{\Rn} |K_{ij} (x',x_n,t)| & dx'dt \leq  \int_0^T \int_{\Rn} | D_{x_n}\Ga (x',x_n,t)|dx'dt  \\
 &+ \int_0^T \int_{\Rn} |L_{ij} (x',x_n,t)|dx'dt.
\end{align*}
 Here, using change of variables ($\frac{x_n^2}t =s$), we get
\begin{align} \label{1006-2}
\begin{array}{ll}\vspace{2mm}
\int_0^T \int_{\Rn}  |D_{x_n} \Ga(x', x_n ,t)| dx' dt
            & = C \int_0^T t^{-\frac32} x_n  e^{-\frac{x_n^2}t}
            \int_{\Rn}  t^{-\frac{n-1}2} e^{-\frac{|x'|^2}t} dx'dt\\ \vspace{2mm}
    & = C  x_n \int_0^T t^{-\frac32}  e^{-\frac{x_n^2}t}
            \int_{\Rn}    e^{-|x'|^2} dx'dt\\
    & = C x_n \int_0^T (\frac{x_n^2}s)^{-\frac32} x_n^2 s^{-2}   e^{-s}ds.
    \end{array}
\end{align}

Hence, to prove Lemma \ref{L-1-norm},
it is sufficient to show  $\int_0^T \int_{\Rn} |L_{ij}(x', x_n, t)| dx'dt < \infty$ for
$1 \leq i \leq n, \,\, 1 \leq j \leq n-1$.

By $\eqref{inequality11}_1$,  for $1 \leq i \leq n$ and $1 \leq j \leq n-1$, we get
\begin{align} \label{0920-8}
\begin{array}{ll}  \vspace{2mm}
\int_0^{x_n^2} \int_{\Rn} |L_{ij}(x', x_n, t)| dx'dt  &\leq C \int_0^{x_n^2}
\int_{\Rn} t^{-\frac12} (|x'|^2 + x_n^2 + t)^{-\frac{n}2} dx'dt\\
& \leq C  \int_0^{x_n^2}  t^{-\frac12} (x_n^2 + t)^{-\frac12} dt =C.
\end{array}
\end{align}
To calculate  $\int_{x_n^2} ^T \int_{\Rn} |L_{ij}(x', x_n, t)| dx'dt$, we  may assume $x_n^2 \leq T$.
By the representation \eqref{0920},  and Lemma \ref{lemma1} and Lemma \ref{lemma2}, we have
\begin{align} \label{0920-6}
\begin{array}{ll} \vspace{2mm}
&\int_{x_n^2} ^T \int_{\Rn} |L_{ij}(x', x_n, t)| dx'dt \\ \vspace{2mm}
& \leq  C \int_{x_n^2} ^T  \int_{\Rn} \int_0^{x_n} t^{-\frac32} y_n e^{-\frac{y_n^2}t}
\Big(    t^{-\frac{n-1}2} e^{-\frac{|x'|^2}t } |x'|^{-1}
      + t^{-\frac{n}2 -\frac12}|x' |e^{-\frac{|x'|^2}{t}}\\
 \vspace{2mm}
 & \qquad +  |x'|^{-n} \int_{|y'| \leq \frac12\frac{|x'|}{\sqrt{t}}}   |y'|^2 e^{-|y'|^2} dy'
 + t^{-\frac{n}2  } \int_{\frac{2 |x'|}{\sqrt{t}} \leq |y'| }
           |y'|^{-n +2}  e^{-|y'|^2} dy'        \Big) dy_n dx' dt \\
 &= I + II + III + IV,
\end{array}
\end{align}
where
\begin{align*}
I& = \int_{x_n^2} ^T  \int_{\Rn} \int_0^{x_n} t^{-\frac32} y_n e^{-\frac{y_n^2}t}
    t^{-\frac{n-1}2} e^{-\frac{|x'|^2}t } |x'|^{-1} dy_n dx' dt ,\\
II& = \int_{x_n^2} ^T  \int_{\Rn} \int_0^{x_n} t^{-\frac32} y_n e^{-\frac{y_n^2}t}
 t^{-\frac{n}2 -\frac12}|x' |e^{-\frac{|x'|^2}{t}}  dy_n dx' dt,\\
III & = \int_{x_n^2} ^T  \int_{\Rn} \int_0^{x_n} t^{-\frac32} y_n e^{-\frac{y_n^2}t}
  |x'|^{-n} \int_{|y'| \leq \frac12\frac{|x'|}{\sqrt{t}}}   |y'|^2 e^{-|y'|^2} dy'  dy_n dx' dt,\\
IV & = \int_{x_n^2} ^T  \int_{\Rn} \int_0^{x_n} t^{-\frac32} y_n e^{-\frac{y_n^2}t}
t^{-\frac{n}2  } \int_{\frac{2 |x'|}{\sqrt{t}} \leq |y'| }
           |y'|^{-n +2}  e^{-|y'|^2} dy'  dy_n dx' dt.
\end{align*}
Using change of variables twice, we have
\begin{align}
\begin{array}{ll} \label{0920-4} \vspace{2mm}
I
&  = \int_{x_n^2}^T t^{-\frac{n}2-1}  \int_{\Rn}  e^{-\frac{|x'|^2}t } |x'|^{-1}  t
         \int_0^{\frac{x_n}{\sqrt{t}}} y_n e^{-y_n^2}  dy_n dx' dt   \\ \vspace{2mm}
&   \leq  C\int_{x_n^2}^T t^{-\frac{n}2-1}
         \int_{\Rn}  e^{-\frac{|x'|^2}t } |x'|^{-1}  t  ( \frac{x_n}{\sqrt{t}} )^2   dx' dt \\ \vspace{2mm}
& = C    \int_{x_n^2}^T  x_n^2 t^{-2} dt\\
& =C
\end{array}
\end{align}
and
\begin{align}
\begin{array}{ll}   \label{0920-5}\vspace{2mm}
 II
&   = \int_{x_n^2}^T t^{-\frac{n}2-2}  \int_{\Rn}  e^{-\frac{|x'|^2}t } |x|  t
            \int_0^{\frac{x_n}{\sqrt{t}}} y_n e^{-y_n^2}   dy_n dx' dt \\ \vspace{2mm}
&   \leq  C\int_{x_n^2}^T t^{-\frac{n}2-2}
              \int_{\Rn}  e^{-\frac{|x'|^2}t } |x'| t  ( \frac{x_n}{\sqrt{t}} )^2   dx' dt \\ \vspace{2mm}
& = C    \int_{x_n^2}^T  x_n^2 t^{-2}   dt \\
& =C.
\end{array}
\end{align}

We divide $III$ into two parts $III = III_1 + III_2$, where
\begin{align*}
III_1 & =\int_{x_n^2} ^T \int_{|x'| \leq \sqrt{t}} \int_0^{x_n} t^{-\frac32} y_n e^{-\frac{y_n^2}t}
  |x'|^{-n} \int_{|y'| \leq \frac12\frac{|x'|}{\sqrt{t}}}   |y'|^2 e^{-|y'|^2} dy'  dy_n dx' dt,\\
III_2 & =\int_{x_n^2} ^T  \int_{|x'| \geq \sqrt{t} } \int_0^{x_n} t^{-\frac32} y_n e^{-\frac{y_n^2}t}
  |x'|^{-n} \int_{|y'| \leq \frac12\frac{|x'|}{\sqrt{t}}}   |y'|^2 e^{-|y'|^2} dy'  dy_n dx' dt.
\end{align*}
Here,
\begin{align*}
III_{1}  & \leq C  \int_{x_n^2}^T  t^{-\frac32}
       \int_{|x'| \leq \sqrt{t}}  |x'|^{-n}   ( \frac{|x'|}{\sqrt{t}}  )^{n +1}
       \int_0^{ x_n}  y_n  e^{-\frac{y_n^2}t}  dy_n dx' dt \\
& = C \int_{x_n^2}^T   x_n^2
         t^{-\frac{n}2 -2}  \int_{|x'| \leq \sqrt{t} }   |x'|   dx' dt  \\
& \leq C \int_{x_n^2}^T   x_n^2 t^{-2}    dt   \\
& =C,\\
III_{2}  & \leq   C\int_{x_n^2}^T   t^{-\frac{3}2}  \int_{|x'| \geq \sqrt{t} } |x'|^{-n}
                            \int_0^{  x_n}   y_n e^{-\frac{y_n^2}t}  dy_n dx' dt  \\
& \leq   C\int_{x_n^2}^T   t^{-\frac{3}2} x_n^2  \int_{|x'| \geq \sqrt{t} }  |x'|^{-n}    dx' dt \\
& \leq  C \int_{x_n^2}^T   t^{-2} x_n^2    dt\\
& =C.
\end{align*}
Hence, we get
\begin{align}\label{0920-2}
III < C.
\end{align}

Similarly, we divide $IV$ into two parts
\begin{align*}
IV
&  = \int_{x_n^2}^T \int_{ |x'| \leq \sqrt{t} }     \int_0^{x_n}
    +   \int_{x_n^2}^T \int_{ \sqrt{t} \leq |x'|}     \int_0^{x_n}\\
&   = IV_1 + IV_2.
\end{align*}
Here, with straightforward integrations
\begin{align*}
IV_1 &  \leq    C  \int_{x_n^2}^T   t^{-\frac{n}2  -\frac32 }   \int_{ |x'| \leq \sqrt{t}}
                     \int_0^{x_n}    y_n e^{-\frac{y_n^2}t} dy_n dx' dt \\
 & \leq C \int_{x_n^2}^T    t^{-\frac{n}2  -\frac32 }  x_n^2 \int_{ |x'| \leq \sqrt{t} }  dx' dt \\
& \leq   C\int_{x_n^2}^T    t^{-2 }  x_n^2    dt\\
& =C
\end{align*}
and
\begin{align*}
IV_2  & \leq C \int_{x_n^2}^T    t^{-\frac{n}2  -\frac32} \int_{ \sqrt{t} \leq |x'|}
                 \int_0^{x_n}  y_n e^{-\frac{y_n^2}t}
          \int_{\frac{2 |x'|}{\sqrt{t}} \leq |y'| }   |y'|^{-n +2}  e^{-|y'|^2} dy'  dy_n  dx' dt \\
    & \leq C \int_{x_n^2}^T    t^{-\frac{n}2  -\frac32} x_n^2   \int_{ \sqrt{t} \leq |x'|}
                    \int_{\frac{2 |x'|}{\sqrt{t}} \leq |y'| }   |y'|^{-n +2}  e^{-|y'|^2}  dy' dx' dt   \\
    & = C\int_{x_n^2}^T    t^{-2} x_n^2   \int_{ 1\leq |x'|}
                    \int_{|x'| \leq |y'| }   |y'|^{-n +2}  e^{-|y'|^2} dy' dx' dt\\
         & \leq C.
\end{align*}
Hence, we get
\begin{align} \label{0920-3}
IV \leq C.
\end{align}

Therefore, from \eqref{0920-8}- \eqref{0920-3}, we prove
\begin{align} \label{0921}
\int_0^T \int_{\Rn} | L_{ij} (x',x_n, t)| dx'dt \leq C
\end{align}
for $1 \leq i \leq n$ and $1 \leq j \leq n-1$,
where $C$ is independent of $x_n$. With \eqref{1006-2}, this implies \eqref{0126-1}.

By the second identity of \eqref{1006-3} and \eqref{0921},  we prove \eqref{0126-2} for the case $i \neq n$,
and by the first identity of \eqref{1006-3} and \eqref{0921},  we prove \eqref{0126-2} for the case $i = n$.
This ends the proof of Lemma 3.1.

\noindent
$\Box$

\noindent
{\bf Proof of Theorem 3.2.}

We begin the proof of Theorem 3.2 by the representation \eqref{representation} of $u$ such that
\begin{align*}
u^i(x,t) &= \sum_{j=1}^{n}\int_0^t \int_{\Rn}
K_{ij}( x^{\prime}-y^{\prime},x_n,t-s)g_j(y^{\prime},s) dy^{\prime}ds,\\
&= \sum_{j=1}^{n-1}\int_0^t \int_{\Rn}
K_{ij}( x^{\prime}-y^{\prime},x_n,t-s)g_j(y^{\prime},s) dy^{\prime}ds\\
&\quad\quad+ \int_0^t \int_{\Rn}
K_{in}( x^{\prime}-y^{\prime},x_n,t-s)g_n(y^{\prime},s) dy^{\prime}ds
\end{align*}
and the last potential for $g_n$ is written as
\begin{align*}
 \int_0^t \int_{\Rn} &
K_{in}( x^{\prime}-y^{\prime},x_n,t-s)g_n(y^{\prime},s) dy^{\prime}ds\\
=&-2 \delta_{in} \int_0^t \int_{\Rn}
D_{x_n} \Ga(x^\prime -y^{\prime},x_n,t-s)g_n(y^{\prime},s) dy^{\prime}ds\\
&+4  \int_0^t \int_{\Rn}
L_{in}( x^{\prime}-y^{\prime},x_n,t-s)g_n(y^{\prime},s) dy^{\prime}ds\\
&-\frac{\partial}{\partial x_i}  \int_0^t \int_{\Rn}
E( x^{\prime}-y^{\prime},x_n)g_n(y^{\prime},s) dy^{\prime}ds.
\end{align*}
Since $L_{in} = L_{ni} +B_{in}$, we have
\begin{align*}
 \int_0^t \int_{\Rn} &
L_{in}( x^{\prime}-y^{\prime},x_n,t-s)g_n(y^{\prime},s) dy^{\prime}ds\\
=& \int_0^t \int_{\Rn}
L_{ni}( x^{\prime}-y^{\prime},x_n,t-s)g_n(y^{\prime},s) dy^{\prime}ds\\
&+ \frac{\partial}{\partial x_i} \int_0^t \int_{\Rn}
\kappa ( x^{\prime}-y^{\prime},x_n,t-s)g_n(y^{\prime},s) dy^{\prime}ds
\end{align*}
 for $1 \leq i \leq n-1$ , where we defined the composite kernel function $\kappa(x,t)$ on ${\bf R}^{n}_{+}\times (0,T)$ by
$$
\kappa (x,t) =   \int_{{\bf R}^{n-1}}  \frac{\partial }{\partial x_n}   \Ga(x^\prime-z^{\prime},x_n,t)
 E  (z^{\prime}, 0)dz^{\prime} .
$$

Define the surface potential ${\bf T}(g_n)$ by
\begin{align}\label{secondpotential}
{\bf T}(g_n)(x,t) =4 \int_{0}^{t }\int_{ {\bf R}^{n-1}} \kappa(x^\prime -y^\prime,x_n,t-s) g_n(y^\prime,s) dy^\prime ds.
\end{align}
Moreover, we have that
$$
\frac{\partial}{\partial x_i}   \int_{\Rn}
E( x^{\prime}-y^{\prime},x_n)g_n(y^{\prime},s) dy^{\prime} = \frac{\partial}{\partial x_i} {\bf S}(g_n)
$$
Therefore we conclude that the tangential part, which is associated with $L_{ij}$, satisfies
$$
 |u_T (x,t) - \nabla {\bf S}_T  (g_n)(x,t) -\nabla  {\bf T}_T  (g_n)(x,t)|
\leq C ||g||_{ L^{\infty} (  {\bf R}^{n-1}\times (0,T ) )  }
$$
for all $(x,t) \in  {\bf R}^{n}_{+}\times (0,T )  $.

The normal velocity $u_n$ behaves even better.
First, we know that $\frac{\partial}{\partial x_n} {\bf S}(g_n)$ is the Poisson kernel expression of the solution for the Laplace equation in the half space and satisfies the maximum principle.
In the case $i=n$, we have a relation  from \eqref{1006-3}
$$
L_{nn} = - \sum_{1 \leq i \leq n-1} L_{ii}  -2D_{x_n} \Ga
$$
which has a bounded $L^1$ norm on the lateral surface. This conclude the maximum modulus estimate of $u_n$.

\noindent
$\Box$

\section{Maximum Modulus Estimate in  $C^2$ Domain}
\setcounter{equation}{0}
We denote the Green's matrix for the domain $\Omega$ by $(G^{\Omega},\theta^{\Omega})$
and for a given point $x \in \Omega$ we let $\bar{x} \in \partial\Omega$ satisfy
$|x-\bar{x}|=dist(x,\partial\Omega)$.
The interior $L^\infty$ bound estimate can be shown by the layer potential method in {\cite{Sh1} and  we consider separately the case
that the generic point $x$ is close enough to $\partial\Omega$.

Indeed, to see the interior boundedness, we need to show the boundedness of the double layer potential in $L^{\infty}(0,T;L^2(\partial\Om))$.
 Since the boundary data is bounded, we can represent the solution by the double layer potential in {\cite{Sh1}
from $L^2$ theory such that
\begin{align}\label{density}
u_i(x,t) = & \int_0^t \int_{\partial\Om} \frac{\partial F_{ij}}{\partial N(y)} (x-y,t-s) h_j(y,s) d\sigma(y)ds\\
&-\int_{\partial\Om} \frac{y_i-x_i}{\omega_n |y-x|^n} h(y,t)\cdot N(y) d\sigma(y)\nonumber\\
=&({\bf K} h)_i (x,t)\nonumber
\end{align}
and
\begin{equation}\label{double}
g = -\frac{1}{2} h + {\bf K} h = (-\frac{1}{2}{\bf I}  + {\bf K} ) h
\end{equation}
in the sense of $L^2( \partial \Om \times (0,T))$ for an $h\in L^2( \partial\Om\times (0,T))$(see Theorem 2.3.6 and Theorem 5.1.2 in \cite{Sh1}). Furthermore $-\frac{1}{2}{\bf I}  + {\bf K}$ is invertible on
$L^2_\sigma( \partial \Om \times (0,T))$, where the subscript $\sigma$ means solenoidal.
 From the representation, we have a
continuity lemma in time of the density function $h$ in (\ref{density}).
\begin{lemm}
The inverse of the double layer potential $-\frac{1}{2}{\bf I}  + {\bf K}$  is bounded in time as an operator of $L^2(\partial\Om)$ and there is a constant $\delta >0$ such that if $|t_1-t_2| <\delta$,
$$
|| h(\cdot,t_2)||_{ L^2( \partial\Om)}  \leq C || g(\cdot,t_1) -g(\cdot,t_2)||_{ L^2( \partial\Om)} + C ||h||_{L^{\infty}(0,t_1;L^2(\Om))}.
$$
By an iteration there is $C$ such that
$$
||(-\frac{1}{2}{\bf I}  + {\bf K})^{-1} g ||_{L^{\infty}(0,T;L^2(\Om))} \leq C||g||_{L^{\infty}(0,T;L^2(\Om))}.
$$
\end{lemm}
\begin{proof}
 We assume the boundary data $g\in L^{\infty}(0,T;L^{2}(\partial\Om))$ and after arranging the singular integrals
in the double layer potential expression we have
\begin{align*}
g_i(x,t_2) -& g_i(x,t_1) =- \frac{1}{2}(h_i(x,t_2)-h_i(x(t-1))\\
&-\int_{\partial\Om} \frac{y_i-x_i}{\omega_n |y-x|^n}( h(y,t_1)-h(y,t_2))\cdot N(y) d\sigma(y)\\
&+ \int_{t_1}^{t_2} \int_{\partial\Om} \frac{\partial F_{ij}}{\partial N(y)} (x-y,t_2-s) h_j(y,s) d\sigma(y)ds\\
&+ \int_{0}^{t_1} \int_{\partial\Om}\left( \frac{\partial F_{ij}}{\partial N(y)} (x-y,t_2 -s) -\frac{\partial F_{ij}}{\partial N(y)} (x-y,t_1 -s)\right) h_j(y,s) d\sigma(y)ds\\
=& (-\frac{1}{2}{\bf I}  + {\bf H} )( h(\cdot,t_2) -h(\cdot,t_1)) + {\bf E}_1 h  + {\bf E}_2  h
\end{align*}
for almost all $0<t_2<t_2<T$ and $x\in\partial\Om$.

We claim $ -\frac{1}{2}{\bf I}  + {\bf H} :L^2(\partial\Om)\rightarrow L^2(\partial\Om)$ is invertible and
$$
 ||(-\frac{1}{2}{\bf I}  + {\bf H} )^{-1} e||_{L^2(\partial\Om)}  \leq C ||e||_{L^2(\partial\Om)}
$$
for a constant $C$. First of all, if we set $e=(-\frac{1}{2}{\bf I}  + {\bf H} ) f$ and consider the normal components, then we have
$$
e_N = (-\frac{1}{2}{\bf I}  + N\cdot {\bf H} )f_N,
$$
where  $N\cdot{\bf H}$ is the standard double layer potential operator of Laplace equation and  $-\frac{1}{2}{\bf I}  +N\cdot {\bf H} $ is invertible. So given vector valued function $e\in L^2(\partial\Om)$, there is a scalar valued function $w\in L^2(\Om)$ satisfying
$$
e_N = (-\frac{1}{2}{\bf I}  + N\cdot {\bf H} )w
$$
with $ ||w||_{L^2(\partial\Om)} \leq C||e||_{L^2(\partial\Om)}$.
Here $w$ is the normal component of $f$ and the tangential component $v$ of $f$ is obtained by
$$
v = -2 (e - e_N N) - 2({\bf H}w - (N\cdot{\bf H})w N).
$$
Therefore we get
$$
f = v + wN
$$
and $f$ satisfies
$$
 ||f||_{L^2(\partial\Om)} \leq C||e||_{L^2(\partial\Om)}.
$$

It remains the estimate ${\bf E}_1 h $ and ${\bf E}_2 h $. Since $\Om$ is $C^2$ domain, in the case of Gaussian kernel, there is $C$ such that for all $(x,y,t) \in \partial\Om\times\partial\Om \times (0,T)$
$$
|\frac{\partial}{\partial N(y)} \Ga (x-y,t) | \leq C \frac{|x-y|^2}{\sqrt{t}^{n+2}} e^{-\frac{|x-y|^2}{2t}}.
$$
Therefore we get from Minkowski inequality and Young's convolution inequality
\begin{align*}
&\left(\int_{\partial\Om} \left| \int_{t_1}^{t_2} \int_{\partial\Om} \frac{\partial \Ga}{\partial N(y)} (x-y,t_2-s) h(y,s) d\sigma(y)ds \right|^2 d\sigma(x) \right)^{\frac{1}{2}}\\
&\leq C\left( \int_{\partial\Om}  \left| \int_{t_1}^{t_2} \frac{1}{\sqrt{t_2 -s}}\frac{1}{\sqrt{t_2-s}^{n-1} } \int_{\partial\Om}
 \frac{|x-y|^2}{t_2-s} e^{-\frac{|x-y|^2}{2(t_2 -s)}}| h(y,s)| d\sigma(y)ds \right|^2 d\sigma(x) \right)^{\frac{1}{2}}\\
&\leq C  \int_{t_1}^{t_2}  \frac{1}{\sqrt{t_2-s}}\left( \int_{\partial\Om}\left| \frac{1}{\sqrt{t_2-s}^{n-1}}  \int_{\partial\Om}
\frac{|x-y|^2}{t_2-s}  e^{-\frac{|x-y|^2}{2(t_2-s)}}| h(y,s)| d\sigma(y)\right|^2  d\sigma(x) \right)^{\frac{1}{2}}ds\\
&\leq C \int_{t_1}^{t_2}  \frac{1}{\sqrt{t_2-s}} ||h(\cdot,s)||_{L^2(\Om)} ds\\
&\leq C\sqrt{ t_2- t_1} ||h||_{L^{\infty}(t_1,t_2;L^2(\Om))}.
\end{align*}
By the same token, assuming $ ||h||_{L^{\infty}(0,t_1;L^2(\Om))}$ is bounded, we have that
$$
|| {\bf E}_2 h(\cdot,t_1)||_{L^2(\partial\Om)} \leq C ||h||_{L^{\infty}(0,t_1;L^2(\Om))}.
$$

\end{proof}

We let  the generic point $x$ be away from the boundary, say $dist(x,\partial\Om) =r_0 >0$.
Since the kernel of the double layer is bounded by $\frac{C}{r_0^{n-1-\epsilon}}$ for each $\epsilon >0$ and the density function $h$ of $g$ for the double layer potential  is bounded in $L^\infty(0,T;L^2(\partial\Om))$,
the interior $L^\infty$ estimate follows.
\begin{coro}
Suppose the boundary data $g\in L^\infty(0,T;L^2(\partial\Om))$. If $dist(x,\partial\Om) \geq r_0 >0$,  $\epsilon >0$ and $t<T$,
then there is $C$ such that
$$
|u(x,t)| \leq \frac{C}{ r_0^{n-1-\epsilon}} ||g||_{L^\infty(0,T;L^2(\partial\Om))}.
$$
\end{coro}

Now we start the boundary estimate.
Since Stokes equations is translation and rotation invariant, we assume that $\bar{x}=0$ and
$x =(0,x_n),\,\, x_n >0$.
If $x$ is close enough to $\partial\Om$, there is a ball $B_r(0)$ centered at origin and $C^2$ function
 $\Phi :{\bf R}^{n-1} \rightarrow {\bf R}$ such that $\Om\cap B_r(0) = \{ x_n > \Phi(x^{\prime})\} \cap B_r(0)$. Furthermore, $\Phi$ satisfies that
\begin{align}\label{flatboundary}
|\Phi(x')| \leq C |x'|^2,\quad
|\na^{\prime}\Phi (x')|\leq C|x'|, \quad
| \na^{\prime}\na^{\prime}\Phi (x^{\prime})|  \leq C
\end{align}
for $ x^\prime\in B^{'}_r(0)$ and the outward unit normal vector $N(x^{\prime}, \Phi(x'))$ at
$(x^\prime,\Phi(x^\prime))\in \pa \Om \cap B_r(0)$ is
$$
N(x^{\prime},\Phi(x^\prime)) = \frac{1}{ \sqrt{1+|\nabla^{\prime}\Phi(x^\prime)|^2}} (\nabla^\prime \Phi(x^\prime), -1).
$$

We define a transform  $ \mu :  \Om \cap B_r(0)   \rightarrow {\bf R}^{n}_{+} $ such that
$$
\mu(y)=\mu(y^\prime,y_n) = (y^\prime, y_n -\Phi(y^\prime))
$$
and note that $\mu^{-1}(y^\prime,y_n) =  (y^\prime, y_n +\Phi(y^\prime))$. Since our generic point $x$
is $(0,x_n)$, we have $\mu(x)=x$. Hence the Green's matrix
 $G$ on the half space can be transformed to a function  $\mu G$ on
 $\Omega$ such that
$$
\mu G (x,y,t) =  G(\mu(x),\mu (y),t)= G (x^\prime,x_n -\Phi(x^\prime), y^\prime,y_n -\Phi(y^\prime)  ,t)
$$
and satisfies the zero boundary condition
$$
\mu G (x,y^\prime,\Phi(y^\prime),t) =0.
$$
Moreover, the transformed Green's matrix $( \mu G,\mu \theta)$
satisfies a perturbed Stokes equations in $\Om\times (0,T)$
\begin{align*}
& \frac{\partial}{\partial t} ( \mu G)_{ij} -\Delta_{y}( \mu G)_{ij}
+\frac{\partial}{\partial y_j} (\mu \theta)_i \\
& = \delta_{ij}\delta(x-y)\delta(t)
    +D_{y_n}(\mu G)_{ij} \Delta^{\prime}\Phi
         +2   D_{y_k}D_{y_n}( \mu G)_{ij}  D_{y_k}  \Phi \\
    & \qquad  -  D_{y_n}D_{y_n}( \mu G)_{ij} D_{y_k}   \Phi  D_{y_k}   \Phi
    -  D_{y_n}( \mu \theta )_i D_{y_j} \Phi   \\
&= \delta_{ij}\delta(x-y)\delta(t) + R(x,y,t)
\end{align*}
and the solenoidal condition
$$
D_{y_j}  ( \mu G)_{ij} = - D_{y_n}( \mu G )_{ij} D_{y_j} \Phi =S_i(x,y,t).
$$

Therefore, if we let the perturbation
$(J,\eta) = (G^{\Om} - \mu G,\theta^{\Om} - \mu \theta)$, then $(J,\eta)$ satisfies the perturbation equations:
\begin{align}\label{perturbation}
\frac{\partial}{\partial t} J_{ij}(x,y,t)   -&\Delta_{y}   J_{ij}(x,y,t)
+\frac{\partial}{\partial y_j}  \eta_{j}(x,y,t)
= R_{ij}(x,y,t)
\end{align}
\begin{align}\label{perturbation1}
D_{y_j}&  J_{ij} (x,y,t) = S_i(x,y,t),
\end{align}
where ${R}$ is
\begin{align*}
{R}_{ij} = & -D_{{y}_n}({\mu} G)_{ij} \Delta^{\prime}_{{y}}\Phi
- 2   D_{{y}_k}D_{{y}_n}( {\mu} G)_{ij}  D_{{y}_k}  \Phi  \\
&+ D_{{y}_n}D_{{y}_n}( {\mu} G)_{ij} D_{{y}_k}
 \Phi   D_{{y}_k}   \Phi
+  D_{{y}_n}( {\mu} \theta )_i D_{{y}_j} \Phi \\
=&I+II+III+IV.
\end{align*}

We have already discussed the boundedness of velocity $u$ in the interior by double layer potential in $L^2$ theory,  we begin to prove the boundedness near the boundary.

The plan to get $L^1$ bound of the perturbation $J$ of Poisson kernel on $\partial\Omega\times(0,T)$ relies on the
$L^p (0,T;W^{2,p}(\Omega))$ estimate and the trace theorem for it.
Recall that the Poisson kernel is a derivative of Green's matrix and that is the reason that
we need $L^p (0,T;W^{2,p}(\Omega))$ Sobolev type estimate. Therefore, we need to estimate the
$L^p$ norm of ${R}$ in (\ref{perturbation}), $L^p(0,T;W^{1,p}(\Om))$ norm of ${S}$ and $L^p(0,T;W^{-1,p} (\Om))$ norm of ${S}_{t}$
in $(B_1\cap{\Om})\times (0,{T})$ in (\ref{perturbation1}), where $W^{-1,p}(\Om)$ is the dual space of $W^{1,p}(\Om)$.

Since the Green's matrix $G$ is associated with the Gaussian kernel
and the composite kernel $H$, we estimate their derivatives first.
We have
\begin{align*}
|D_{y_n}\Ga(x-y,t)| & \leq \frac{C}{\sqrt{t}^n}\frac{|y-x|}{t}e^{-\frac{|y^\prime|^2+|y_n-x_n|^2}{2t}} \\
|D_{y_n} \Ga(x-y^{*},t)| & \leq\frac{C}{\sqrt{t}^n}\frac{|y^{*}-x|}{t}
e^{-\frac{|y^\prime|^2+|y_n+x_n|^2}{2t} }\\
|D_{y_k}D_{y_n}\Ga(x-y,t)| & \leq \frac{C}{\sqrt{t}^n}\frac{|y-x|^2 }{t^2}e^{-\frac{|y^\prime|^2+|y_n-x_n|^2}{2t}} \\
|D_{y_k}D_{y_n} \Ga(x-y^{*},t)| & \leq\frac{C}{\sqrt{t}^n}\frac{|y^{*} -x|^2 }{t^2}
e^{-\frac{|y^\prime|^2+|y_n+x_n|^2}{2t} }.
\end{align*}

Since  $|D_{y_k} \Phi(y^\prime) | \leq C |y^\prime|$,  $|\Delta^\prime \Phi(y^\prime) | \leq C$ and $x^\prime =0$, we get
\begin{align*}
|D_{y_n}\Ga(x-y,t) \Delta^\prime \Phi(y^\prime)|& \leq
 \frac{C}{\sqrt{t}^{n+1}}\frac{|y-x|}{\sqrt{t}}e^{-\frac{|y^\prime|^2+|y_n-x_n|^2}{2t}} \in
 L^{p}((\Omega\cap B_r) \times (0,T))\\
|D_{y_k}D_{y_n}\Ga(x-y,t) \nabla^\prime \Phi(y^\prime)   |& \leq
 \frac{C}{\sqrt{t}^{n+1}}\frac{|y-x|^3}{\sqrt{t}^3}e^{-\frac{|y^\prime|^2+|y_n-x_n|^2}{2t}} \in
 L^{p}((\Omega\cap B_r) \times (0,T))
\end{align*}
as a function of $y$ for all $p\in [1,\frac{n+2}{n+1})$. In the same way, we have
$$
D_{y_n}\Ga(x-y^{*},t) \Delta^\prime \Phi(y^\prime) ,\quad
 D_{y_k}D_{y_n}\Ga(x-y,t) \nabla^\prime \Phi(y^\prime)  \in
 L^{p}((\Omega\cap B_r) \times (0,T))
$$
as a function of $y$ for all $p\in [1,\frac{n+2}{n+1})$.

Applying (\ref{inequality11}), we have
\begin{align*}
\big|D_{y_n}D_{x_j}\int_{0}^{x_n}\int_{ {\bf R}^{n-1}} D_{x_i} E(x-z)\Ga(z-y^{*},t) dz\big|
& \leq  \frac{C}{t^{\frac12} (|x'-y'|^2 +|y_n +x_n|^2+t )^{\frac12 n} }\\
\big|  D_{y_k} D_{y_n}D_{x_j}\int_{0}^{x_n}\int_{ {\bf R}^{n-1}} D_{x_i} E(x-z)\Ga(z-y^{*},t) dz\big|
 & \leq  \frac{C}{t^{\frac12} (|x'-y'|^2 +|y_n +x_n|^2+t )^{\frac12( n+1)} }.
\end{align*}
Hence, we have for $p\in [1,\frac{n+2}{n+1})$
\begin{align*}
\big|D_{y_n}D_{x_j}\int_{0}^{x_n}\int_{ {\bf R}^{n-1}} D_{x_i}& E(x-z)\Ga(z-y^{*},t) dz
  \Delta^\prime \Phi(y^\prime)\big| \\
& \leq  \frac{C}{\sqrt{t}^{n+1}}\frac{1}{ \sqrt{ (|\frac{|y^{*}-x|^2}{t} +1 )}^{n }} \in  L^{p}((\Omega\cap B_r) \times (0,T))
\end{align*}
\begin{align*}
\big|  D_{y_k} D_{y_n}D_{x_j}\int_{0}^{x_n}\int_{ {\bf R}^{n-1}} D_{x_i}& E(x-z)\Ga(z-y^{*},t) dz
 \nabla^\prime \Phi(y^\prime) \big| \\
 & \leq   \frac{C}{\sqrt{t}^{n+1}}\frac{1}{ \sqrt{ (|\frac{|y^{*}-x|^2}{t} +1 )}^{n }} \in  L^{p}((\Omega\cap B_r) \times (0,T)).
\end{align*}

Although there is a transformation $\mu$ of domain, these estimates imply that $I$, $II$  and $III$ are in
$ L^{p}((\Omega\cap B_r) \times (0,T))$
as a function of $y$ for all $p\in [1,\frac{n+2}{n+1})$.

It remains to get $L^p$ estimate of the pressure kernel $\theta$. For each fixed time $t$, we have
\begin{align} \label{0127}
\notag|IV| \leq& C |y^\prime| |\int_{ {\bf R}^{n-1}}  D_{x_i} E(x'-z', x_n)
\frac{y_n}{t}\frac{1}{\sqrt{t}^n} e^{- \frac{|z^\prime-y^\prime|^2 +y_n^2}{2t}}dz^{\prime}|\\
&+ C |y^\prime| \left| D_{y_n}D_{y_n}\int_{ {\bf R}^{n-1}} \frac{1}{\sqrt{|x^\prime-z^\prime|^2 + x_n^2 }^{n-2}}
\frac{1}{\sqrt{t}^n} e^{- \frac{|z^\prime-y^\prime|^2 +y_n^2}{2t}}dz^{\prime}\right|.
\end{align}
The first term on the right is $L^p$ for $p\in [1,\frac{n+2}{n+1})$
by the Young's convolution inequality
since the kernel $ \frac{x_n}{\sqrt{|z^\prime|^2 + x_n^2 }^{n}}$ has bounded $L^1(  {\bf R}^{n-1})$
estimate as a function of $z^\prime$ independent of $x_n$.
  For the second term, we recall the following proposition by
Solonnikov (Proposition 2.3 in \cite{So1}):
\begin{lemm}\label{0127-2}
 Let M(x, t) be a function defined for $x \in {\bf R}^{n}_{+} $ and $ t > 0$ and
having the properties
$$
M(\lambda x, \lambda^2 t) = \lambda^{m}M(x, t) ,\quad \lambda > 0,
$$
$$
|D^{k}_{x} D^{s}_{t} M(x, t)| \leq C t^{\frac{ m-k-2s}{2}} \exp\left( -\frac{|x|^2}{2t}\right).
$$
Then the integral
$$
J(x, y_n, t) = \int_{{\bf R}^{n-1}} E(y) M(x^\prime-y^\prime,x_n,t) dy^\prime
$$
satisfies the conditions
$$
J(\lambda x, \lambda y_n, \lambda^2 t) = \lambda^{m+1}J(x, y_n, t),
$$
$$
|D^k_{x}D^{l}_{y_n}D^s_{t} J(x,y_n,t) |
\leq C t^{ \frac{m+n-1-2s-k_n}{2} }\left( |x^\prime|^2+(x_n+y_n)^2 +t\right)^{- \frac{|k^{\prime}|+l+n-2}{2} }
e^{ -\frac{x_n^2}{2t}},
$$
where $k = (k_1, . . ., k_n)$ and $|k^\prime|=k_1+\cdot\cdot\cdot+k_{n-1}$.
\end{lemm}
See Proposition 2.3 in \cite{So1}.

So, we find the second term  of \eqref{0127} without the transformation $\mu$ is bounded by
\begin{align*}
 C\frac{|y^\prime|}{\sqrt{t} } &
 \left( |\frac{y^\prime}{\sqrt{t}}|^2+(\frac{x_n+y_n}{\sqrt{t}})^2 +1\right)^{- \frac{n}{2} }
\frac{1}{\sqrt{t}^n}e^{ -\frac{y_n^2}{2t}}\\
 & \leq C\big( \frac{r}{\sqrt{t}}+1\big) \frac{|y^\prime|}{\sqrt{t} }
 \left( |\frac{y^\prime}{t}|^2+(\frac{x_n+y_n}{t})^2 +1\right)^{- \frac{n-1}{2} }
\frac{1}{\sqrt{t}^n}e^{ -\frac{y_n^2}{2t}}
\end{align*}
which is in $L^p((\Om\cap B_r)\times(0,T))$ for $p\in (1,\frac{n+2}{n+1})$.

This concludes that $IV$ is in $L^p$ and $R$ in (\ref{perturbation}) is in $L^p$ for all $p\in (1,\frac{n+2}{n+1})$ after adjustment of the domain transformation $\mu$.

To get $L^p(0,T;W^{1,p}(\Om\cap B_r))$ bound of $S$ in (\ref{perturbation1})
we follow a similar program to $R$.
 Indeed, we have
$$
\nabla_y S =  -\nabla_y D_{y_n}( \mu G )_{ij} D_{y_j} \Phi (y^\prime)
-  D_{y_n}( \mu G )_{ij} \nabla_y D_{y_j} \Phi (y^\prime).
$$
The terms in the right hand side have already been considered in the estimates of
$I$, $II$ and $III$ of $R$ except
$D_{y_n} D_{y_n}( \mu G )_{ij} D_{y_j} \Phi (y^\prime)$.
But,  $ D_{y_n} D_{y_n}( \mu G )_{ij}  =
\sum_{1 \leq j \leq n-1} D_{y_k} D_{y_n}( \mu G_{ij}) $ and hence
$D_{y_n} D_{y_n}( \mu G )_{ij} D_{y_j} \Phi (y^\prime)$ has the form of $II$.
Therefore, we get
$$
||S||_{L^p(0,T;W^{1,p}(\Om\cap B_r))} < C \quad\mbox{ independent of} \quad  x.
$$

It remains to find $L^p(0,T;W^{-1,p}(\Om\cap B_r))$ estimate of $D_t S$.
Since $S$ is defined as
$$
S_i(x,y,t) =  - D_{y_n}( \mu G )_{ij} D_{y_j} \Phi (y^\prime)
$$
and $\Phi$ is independent of $y_n$, $L^p (0,T;W^{-1,p}(\Om\cap B_r))$ norm  of $S_t$  is bounded  by
$$
C \int_{0}^{T} \int_{ \Om\cap B_r} | D_t( \mu G) \nabla^{\prime}  \Phi (y^\prime)|^p dy dt
$$
for a constant $C$. By disregarding $\Phi$, we have
\begin{align*}
D_t  G_{ij} &= \delta_{ij}\big( D_t \Ga(x-y,t) -  D_t \Ga(x-y^*,t)\big)\\
& +4(1-\delta_{jn}) D_{x_j} \int_{0}^{x_n} \int_{{\bf R}^{n-1}} D_{x_i} E(x-z)  D_t \Ga(z-y^{*},t) dz
\end{align*}
and
\begin{align*}
|D_t  G_{ij}| &\leq C \left( \frac{1}{\sqrt{t}^{n+2}} +  \frac{|x-y|^2}{\sqrt{t}^{n+4}} \right)
e^{ -\frac{ |x-y|^2}{2t}} +   C \left( \frac{1}{\sqrt{t}^{n+2}} +  \frac{|x-y^{*}|^2}{\sqrt{t}^{n+4}} \right)
e^{ -\frac{ |x-y^{*}|^2}{2t}}\\
& + C\left| D_{x_j} \int_{0}^{x_n} \int_{{\bf R}^{n-1}} D_{x_i} E(x-z)  D_t \Ga(z-y^{*},t) dz\right|
\end{align*}
and Proposition 2.5 in \cite{So1},
we have
$$
\left| D_{x_j} \int_{0}^{x_n} \int_{{\bf R}^{n-1}} D_{x_i} E(x-z)  D_t \Ga(z-y^{*},t) dz\nabla^{\prime}\Phi(y^\prime) \right|
 \leq  \frac{C |y^\prime| }{t  (|x-y^{*}|^2 +t )^{\frac{n}2 } } e^{ - \frac{y_n^2}{2t}}.
$$
This implies that $D_t  G_{ij}\nabla^{\prime}\Phi(y^\prime) \in L^p(( \Om\cap B_r)\times (0,T))$
for all $p\in [1,\frac{n+2}{n+1})$.

Since the estimates of $R$ and $S$ hold only in a small ball near boundary, we need a localization.
For the localization, we take a cut off function $\phi$ such that $\phi =1$ in $B_r$ and $\phi=0$ in the complement of
$B_{2r}$ and we consider $(\phi J, \phi\eta)$ as a solution to the inhomogeneous Stokes equations. We delete the generic point $x$ in the various expressions.
Therefore by Theorem 3.1 in \cite{So1}, we obtain the following lemma for the perturbation $(J,\eta)$.
\begin{lemm}
There is a constant $C$ depending on $r$ and $\Om$ such that
\begin{align*}
|| J||_{L^p(0,T; W^{2,p}(\Om\cap B_r))}& + || \eta||_{L^p(0,T;W^{1,p}(\Om\cap B_r) )} \\
&\leq C(1+ ||G^\Om||_{L^p(0,T; W^{1,p}(\Om\cap(B_{2r}\setminus B_r)))} +
 || \theta^\Om||_{L^p(\Om\cap(B_{2r}\setminus  B_r) \times (0,T))} )
\end{align*}
for all $p \in (1,\frac{n+2}{n+1})$.
\end{lemm}
By the trace theorem in $W^{1,p}(\Om\cap B_r)$, the following lemma also holds.
\begin{lemm}
There is a constant $C$ depending on $r,\Om$ such that
\begin{align*}\label{trace}
||\nabla  J&||_{L^p ( 0,T; W^{1-\frac{1}{p},p}(\partial \Om\cap B_r)) }
+ || \eta||_{L^p(0,T; W^{1-\frac{1}{p},p} (\partial\Om \cap B_r)  } \\
&\leq C(1+ ||G^\Om||_{L^p(0,T; W^{1,p}(\Om\cap(B_{2r}\setminus B_r)))} +
 || \theta^\Om||_{L^p(\Om\cap(B_{2r}\setminus  B_r) \times (0,T))} )
\end{align*}
for all $p \in (1,\frac{n+2}{n+1})$.
\end{lemm}
The generic point $x$ is in $B_r$ and hence the Green's matrix $(G^\Om,\theta^\Om)$ has no singularity in the complement of $B_r$ as a function of $(y,t)$.
Therefore we have that for all $p\in [1,\infty]$
$$
 ||G^\Om||_{L^p(0,T; W^{1,p}(\Om\cap(B_{2r}\setminus B_r)))} +
 || \theta^\Om||_{L^p(\Om\cap(B_{2r}\setminus  B_r) \times (0,T))} \leq C
$$
for a constant $C$ depending only on $p, r$ and $\Om$.

Now we prove our main theorem. The Poisson kernel $K^\Om(x,y,t)$ satisfies
$$
K^\Om(x,y,t) = \frac{\partial}{\partial N(y)} G^\Om(x,y,t) - \theta^\Om(x,y,t)N(y), \quad\mbox{for all }
\quad(x,y,t)\in\Om\times\partial\Om\times(0,T).
$$
We have that
$$
G^\Om = \mu G +  J, \quad \theta^\Om =  \mu \theta  + \eta
$$
and from Lemma 4.3  we know that $\nabla J$ and $\eta$ have bounded
 $L^1( ( \partial \Om\cap B_r) \times  (0,T) )$  norms independent of $x$ since $L^p ( 0,T; W^{1-\frac{1}{p},p}(\partial \Om\cap B_r))$ for $p\in (1,\frac{n+2}{n+1})$ is embedded in $L^1( ( \partial \Om\cap B_r) \times  (0,T) )$.
So we need to consider only $ \frac{\partial}{\partial N(y)}\mu G(x,y,t)$
and $ \mu \theta (x,y,t)$.
The $L^1( ( \partial \Om\cap B_r) \times  (0,T) )$ bound of $ \mu \theta (x,y,t)$ as a function of $y^\prime$ for the generic point $x=(0,x_n)$ can be obtained by Lemma 4.1 after considering coordinate transform $\mu$.
From the definition of the transformation of $\mu$ and the local representation of the boundary $\partial\Om$, we have that for $y=(y^\prime,y_n)=(y^\prime,\Phi(y^\prime)) \in \partial\Om\cap B_r$
\begin{align*}
 \frac{\partial}{\partial N(y)} G(x,y,t) = & \frac{1}{\sqrt{ 1+|\nabla^\prime \Phi(y^\prime)|^2}}
\nabla^{\prime}_{y^\prime} G (x,y^\prime,y_n -\Phi(y^\prime),t ) \cdot\nabla^{\prime}_{y^\prime}\Phi(y^\prime)\\
& - \frac{1}{\sqrt{ 1+|\nabla^\prime \Phi(y^\prime)|^2}}  \frac{\partial}{\partial y_n}  G (x,y^\prime,y_n -\Phi(y^\prime),t )\\
= & \frac{1}{\sqrt{ 1+|\nabla^\prime \Phi(y^\prime)|^2}}
\nabla^{\prime}_{y^\prime} G (x,y^\prime,0,t) \cdot\nabla^{\prime}_{y^\prime}\Phi(y^\prime)\\
& -  \frac{1}{\sqrt{ 1+|\nabla^\prime \Phi(y^\prime)|^2}} \frac{\partial}{\partial y_n}  G (x,y^\prime,0,t).
\end{align*}
Furthermore, we have already proved that the $L_1( ( \partial \Om\cap B_r) \times  (0,T) )$ norm estimate of the first term
$$
\int_{0}^{T}\int_{ \partial\Om \cap B_r}  \left|  \frac{1}{\sqrt{ 1+|\nabla^\prime \Phi(y^\prime)|^2}}
\nabla^{\prime}_{y^\prime} G (x,y^\prime,0,t) \cdot\nabla^{\prime}_{y^\prime}\Phi(y^\prime)\right| dy^\prime dt \leq C
$$
for some $C$ independent of $x$ since $| \nabla^{\prime}\Phi(y^\prime)|\leq c|y^\prime|$.

 By the expression of Poisson kernel $K$ we have
\begin{align*}
 - \frac{\partial}{\partial y_n}  G_{ij} & (x,y^\prime,0,t) =  K_{ij}  (x^\prime-y^\prime,x_n,t) + \delta_{jn} \eta_{i}  (x^\prime-y^\prime,0,t) \\
&=-2\delta_{ij} D_{x_n} \Ga(x^\prime-y^\prime,x_n,t)
+ 4( L_{ij}(x^\prime-y^\prime,x_n,t) - \delta_{jn} B_{in} (x^\prime-y^\prime,x_n,t))\\
&\quad + 4 \delta_{jn} B_{in} (x^\prime-y^\prime,x_n,t) - \delta_{jn}\delta(t) D_{x_i} E(x^\prime-y^\prime,x_n)
\end{align*}
We know already that $-2\delta_{ij} D_{x_n} \Ga(x^\prime-y^\prime,x_n,t)
+ 4( L_{ij}(x^\prime-y^\prime,x_n,t)   - \delta_{jn}B_{in} (x^\prime-y^\prime,x_n,t))$  has $L^1$ bounded norm as a function of $(y^\prime,t)$.
Therefore  in the solution expression for of $u$ we can write
$$
\int_{0}^{t}\int_{\partial\Omega \cap B_r}  \delta_{jn}\delta(t-s) D_{x_i} E(x^\prime-y^\prime,x_n) g_j (y,s) d\sigma_{y}
$$
$$
=\frac{\partial}{\partial x_i} \int_{\partial\Omega \cap B_r}   E(x^\prime-y^\prime,x_n) g_n (y,t) d\sigma_{y}
$$
and
$$
4\int_{0}^{t}\int_{\partial\Omega \cap B_r}  \delta_{jn}  B_{in} (x^\prime-y^\prime,x_n,t-s) g_j (y,s) d\sigma_{y}
$$
$$
= \frac{\partial}{\partial x_i} {\bf T} (g_n)(x,t).
$$
If we denote $e_n=(0,1)$ which is the coordinate vector for $y_n$, we have that the component of boundary data $g_n$ is
$$
g_n = g \cdot  N(y) + g\cdot(e_n -N(y)) \quad\mbox{for}\quad y \in \partial\Om\cap B_r
$$
where $g=(g_1,g_2,\cdot\cdot\cdot,g_n)$ is the boundary data and hence we have
\begin{align*}
\nabla_x  \int_{\partial\Omega \cap B_r}   E(x^\prime &-y^\prime,x_n) g_n (y,t) d\sigma_{y}
=\nabla_x {\bf S}(g\cdot N X^{\partial\Om\cap B_r}) \\
+& \nabla_x  \int_{\partial\Omega \cap B_r}(   E(x^\prime-y^\prime,x_n)
-  E(x^\prime-y^\prime,x_n -\Phi(y^\prime))) g(y,t)\cdot (e_n-N(y))  d\sigma_{y}\\
+& \nabla_x  \int_{\partial\Omega \cap B_r}   E(x^\prime-y^\prime,x_n) g(y,t)\cdot (e_n-N(y))  d\sigma_{y},
\end{align*}
where ${\bf S}$ is the single layer potential operator and  $X$ is the characteristic function. Since $x^\prime =0$ and $\Phi(y^\prime) \leq C |y^\prime|^2$,
$$
 \int_{\partial\Omega \cap B_r} \left|  \nabla_x (   E(x^\prime-y^\prime,x_n)
-  E(x^\prime-y^\prime,x_n -\Phi(y^\prime)))\right| d\sigma_y \leq C.
$$
Then, by observing that
$$
|e_n-N(y)| \leq C |y^\prime |
$$
we find that $\nabla_x  E(x^\prime-y^\prime,x_n) \cdot (e_n-N(y)) $
 has  bounded $L^1$ norm as a function of $y^\prime$ and
we have that
$$
\sup_{x\in \Om\cap B_r}
\left|  \nabla_x  \int_{\partial\Omega \cap B_r}   E(x^\prime-y^\prime,x_n) g(y,t)\cdot (e_n-N(y))  d\sigma_{y}\right| \leq C || g ||_{L_{\infty}(\partial\Om\times (0,T))}.
$$
Similarly we find that $\nabla \kappa(x^\prime-y^\prime,x_n,t)  (e_n-N(y)) $ has a bounded $L^1$ norm as a function of $(y^\prime,t)$ and
we have that
$$
\sup_{(x,t) \in \Om\cap B_r \times (0,T) }
\left|   \int_{0}^{t}  \int_{\partial\Omega \cap B_r}   \nabla\kappa(x^\prime-y^\prime,x_n,t-s) g(y,s)\cdot (e_n-N(y))  d\sigma_{y}ds \right|
$$
$$
 \leq C || g ||_{L^{\infty}(\partial\Om\times (0,T))}.
$$

With the interior $L^\infty$ estimate, localization with the small balls $B_r$ and the preceding kernel estimates of $L^1$ bound,
 we prove our main Theorem 1.1.

For Corollary 1.2, we observe that
$$
|\nabla_x {\bf S}(g\cdot N X^{\partial\Om\cap B_r}) |
\leq C\int_{\Om\cap B_r}  \frac{ 1}{ | y^\prime |^{n-1} } | g(y^{\prime},\Phi(y^\prime),t) \cdot N(y^\prime,\Phi(y^\prime)  | dy^\prime,
$$
where $x = (0,x_n)$.  Similarly, we also have
$$
|\nabla_x {\bf T}(g\cdot N X^{\partial\Om\cap B_r}) |
\leq C\int_0^T \int_{\Om\cap B_r}  \frac{ 1}{ | y^\prime |^{n-1} } | g(y^{\prime},\Phi(y^\prime),t) \cdot N(y^\prime,\Phi(y^\prime))  | dy^\prime dt.
$$
The boundedness follows from the Dini-continuity of $g\cdot N$.

\noindent
$\Box$

\noindent
{\bf Acknowledgment.} The research is supported by KRF2011-028951.

\end{document}